\documentclass[11pt]{amsart}

\usepackage[top=1in, bottom=1in, left=1in, right=1in]{geometry}

\usepackage{graphicx}
\usepackage[dvipsnames]{xcolor}
\usepackage{amsfonts,amsmath,amssymb,amsthm,bbm,centernot}
\usepackage{mathrsfs}
\usepackage{enumerate}
\usepackage[normalem]{ulem}
\usepackage{hyperref}
\hypersetup{colorlinks=true,linkcolor=blue!50!black,citecolor=blue!50!black,urlcolor=blue!50!black}
\usepackage{mathtools}
\usepackage{multirow}

\newtheorem{theorem}{Theorem}[section]
\newtheorem{proposition}[theorem]{Proposition}
\newtheorem{lemma}[theorem]{Lemma}
\newtheorem{corollary}[theorem]{Corollary}

\theoremstyle{remark}

\theoremstyle{definition}

\newcommand{\N}{\mathbb N}

\newcommand{\R}{\mathbb R}
\newcommand{\Z}{\mathbb Z}

\renewcommand{\epsilon}{\varepsilon}

\renewcommand{\P}{\mathbb{P}}

\newcommand{\B}{\mathfrak{B}}
\newcommand{\D}{\mathfrak{D}}

\newcommand{\fl}[1]{\lfloor #1 \rfloor}

\newcommand{\ind}[1]{ \mathbbm{1}_{\{ #1 \}} }

\DeclareMathOperator{\Var}{Var}

\title[P\'olya urns for PSR walks]{Extension of results on generalized P\'olya urns for polynomially self-repelling walks}

\author{Elena Kosygina}
\address{Elena Kosygina\\One Bernard Baruch Way \\ Department of Mathematics, Box B6-230 \\ Baruch College \\ New York, NY 10010 \\ USA}
\email{elena.kosygina@baruch.cuny.edu}
\urladdr{http://www.baruch.cuny.edu/math/elenak/}

\author{Laure Mar\^ech\'e}
\address{Laure Mar\^ech\'e\\Institut de Recherche Mathématique Avancée, UMR 7501 Université de Strasbourg et CNRS, 
7 rue René Descartes, 67000 Strasbourg, France}
\email{laure.mareche@math.unistra.fr}
\urladdr{https://irma.math.unistra.fr/~mareche/}

\author{Thomas Mountford}
\address{Thomas Mountford\\\'Ecole Polytechnique F\'ed\'erale de Lausanne\\Department of Mathematics\\EPFL SB MATH PRST\\
MA B1 517 (B\^atiment MA)\\
Station 8\\
CH-1015 Lausanne\\
Switzerland}
\email{thomas.mountford@epfl.ch}
\urladdr{http://people.epfl.ch/thomas.mountford}

\author{Jonathon Peterson}
\address{Jonathon Peterson\\Purdue University\\Department of Mathematics\\150 N University Street\\West Lafayette, IN  47907\\USA}
\email{peterson@purdue.edu}
\urladdr{http://www.math.purdue.edu/~peterson}



\begin{document}

\begin{abstract}
  This is a technical note which extends the results of Kosygina, Mountford and Peterson (Ann. Probab., 51(5):1684–1728, 2023, Lemma 4.1 and Section 4.2) about generalized P\'olya urns from a specific weight
  function $w(n) = (n+1)^{-\alpha}$ to a general family of
  weight functions satisfying
  $(w(n))^{-1}=n^{\alpha}\left(1+2Bn^{-1}+O\left(n^{-2}\right)\right)$
  as $n\to \infty$. The latter was considered by T\'oth (Ann. Probab., 24(3):1324–1367, 1996) as part of his study of polynomially self-repelling
  walks. This extension is used in the study of scaling limits of
  these walks and related processes; see
  \href{https://arxiv.org/abs/2610.28935}{arXiv:2610.28935} and
  \href{https://arxiv.org/abs/2610.38591}{arXiv:2610.38591}.
\end{abstract}

\maketitle

\section{Brief introduction}

Given a weight function $w:\N_0:=\N\cup\{0\}\to (0,\infty)$, a self-interacting random walk $(X_n)_{n\ge 0}$ on $\mathbb{Z}$ starts from the origin and takes steps $X_{n+1}-X_{n} = \pm 1$ with probabilities given by 
\begin{align}
 &\P\left( X_{n+1} = X_n+1 \mid X_i, \, i\leq n \right) \nonumber \\
&\qquad  = 1 - \P\left( X_{n+1} = X_n-1 \mid X_i, \, i\leq n \right)
 = \frac{w(\ell_n(X_n))}{w(\ell_n(X_n)) + w(\ell_n(X_n-1)) }, \label{tp}
\end{align}
where 
\[
 \ell_n(x) = \sum_{i=1}^n \ind{\{X_{i-1},X_i\} = \{x,x+1\} }, \quad n\geq 0, x \in \Z, 
\]
is the number of times the walk has crossed the (undirected) edge $\{x,x+1\}$ by time $n$. If the weight function $w$ is non-increasing then the walk is self-repelling, while if $w$ is non-decreasing then the walk is self-attracting. 

In \cite{tGRK}, among other classes of self-interacting random walks on $\Z$, B.\,T\'oth introduced and studied a family of so-called \textbf{polynomially self-repelling (PSR) walks}. These random walks are characterized by a non-increasing $w$ which has the following asymptotics: for some $\alpha>0$ and $B \in \R$
\begin{equation}\label{generalw}
 \frac{1}{w(n)}=n^{\alpha}\left(1+\frac{2B}{n}+O\left(\frac{1}{n^2}\right)\right)\quad\text{as $n\to \infty$.}
\end{equation}

One of the results in \cite{tGRK} is a generalized Ray-Knight
theorem for PSR walks, that is, a scaling limit for their local times (numbers of crossings of directed edges) when one of these local times reaches some threshold. The limiting
process in that theorem is consistent with convergence of 
the rescaled PSR walk (under the diffusive scaling) to a process known as a
Brownian motion perturbed at its extrema (BMPE) (see \cite{cdPUPBM,pwPBM}). However, a recent
result in \cite{KMP23} shows that for a specific
choice of the weight function
\begin{equation}\label{specificw}
 w(n) = (n+1)^{-\alpha}, \quad n\geq 0, 
\end{equation}
the rescaled PSR walk cannot converge to a BMPE.  This implies
that a generalized Ray-Knight theorem is not sufficient for
identifying the functional limit of a random walk. To obtain this
``negative'' result, it was enough to demonstrate non-convergence for
the walk with the specific choice \eqref{specificw} for the weight
function $w$.  The results of \cite{KMP23} left open the
question of whether PSR walks have a scaling limit and, if they do, what this limit might be. For a ``positive'' result such as a
functional limit theorem, it makes sense to consider PSR walks with general
weight functions \eqref{generalw}, and this question is
addressed in \cite{FLT}. The results of this note are also used in
\cite{Yu26} in the proof of a joint Ray--Knight theorem for PSR walks.
Hence, the purpose of this note is to show that certain results that were proved in \cite{KMP23} for the PSR walk with weight function \eqref{specificw} can, in fact, be extended to the general PSR case with weight function \eqref{generalw}. We note that it was already stated in \cite[Remark 1.6]{KMP23} that the nonconvergence result of \cite{KMP23} mentioned above (Theorem 1.5 there) and its proof remain valid for weight functions satisfying \eqref{generalw}; this note supplies the details for the results on generalized P\'olya urns used in that proof.

While assuming that the weight function $w$ is monotone is natural,
since it makes the walk either self-repelling or self-attracting, it
may be of interest in the future to consider more general models where
the weight function is not monotone.  The results of this note are
valid for any weight function which satisfies \eqref{generalw} and do
not require monotonicity. Note, however, that \eqref{generalw} implies $\frac{1}{w(n+1)}-\frac{1}{w(n)}=\alpha n^{\alpha-1}+O(n^{\alpha-2})$ as $n\to\infty$, so $w$ is eventually strictly decreasing.

Many of the statements below can be proved in the same way as the corresponding results in \cite{KMP23}. However, our more general asymptotics and the lack of monotonicity will often require more careful justifications, and at some points new arguments.

\section{PSR walks and the associated generalized P\'olya urn model}\label{sec:Polya}
 
For a self-interacting walk \eqref{tp}, the sequence of left/right
steps from a given site of $\mathbb{Z}$ can be thought of as being generated by a
generalized P\'olya urn process at that site.  The urn processes
are slightly different depending on whether the site is to the left of
the origin, to the right of the origin, or at the origin and, thus,
following \cite{tGRK}, we first describe the family of generalized
P\'olya urn processes, and then specialize depending on the
  site's location with respect to the origin. To generate a random
  walk path, one can generate independent urn processes at each site,
  and then let the walk take a step to the right when the next ball
  drawn from the urn at the walk's current location is red and a step
  to the left when it is blue.

Given sequences
of positive numbers $(b(i))_{i\geq 0}$ and $(r(i))_{i\geq 0}$, the
generalized P\'olya urn process
$\big((\mathfrak{B}_n, \mathfrak{R}_n)\big)_{n\geq 0}$ is a Markov chain on
$\N_0^2$ started at $(\mathfrak{B}_0, \mathfrak{R}_0) = (0,0)$ with
transition probabilities given by
\begin{align*}
 \P\left( (\mathfrak{B}_{n+1}, \mathfrak{R}_{n+1}) = (i+1,j) \mid (\mathfrak{B}_n, \mathfrak{R}_n) = (i,j) \right) &= \frac{b(i)}{b(i)+r(j)},\ \text{and}\\
\P\left( (\mathfrak{B}_{n+1}, \mathfrak{R}_{n+1}) = (i,j+1) \mid (\mathfrak{B}_n, \mathfrak{R}_n) = (i,j) \right) &= \frac{r(j)}{b(i)+r(j)},\ \ i,j\ge 0.
\end{align*}
If we consider this as being generated by drawing red/blue balls from an urn, then $\mathfrak{B}_n$ and $\mathfrak{R}_n$ will be the numbers of blue and red balls, respectively, drawn from the urn up to time $n$.

\textbf{Rubin's construction.} It will be helpful at times to use an
equivalent construction of a generalized P\'olya urn process (due to
H.\,Rubin, see \cite{Dav90}) using exponential random variables.  Suppose that
$B_1,B_2,\ldots,R_1,R_2,\ldots$ are independent random variables with
$B_i \sim \mathrm{Exp}(b(i-1))$ and $R_i \sim \mathrm{Exp}(r(i-1))$, for
$i\geq 1$.  We then make a red mark on $(0,\infty)$ at
$\sum_{i=1}^k R_i$ for every $k\geq 1$ and similarly a blue mark on
$(0,\infty)$ at $\sum_{i=1}^k B_i$ for every $k\geq 1$.  The urn
process can then be constructed by reading off the sequence of red and
blue marks in order: every red mark corresponds to drawing a red ball  ($\mathfrak{R}$ increases by one) and every blue mark
corresponds to drawing a blue ball ($\mathfrak{B}$
increases by one).

For $* \in \{-,+,0\}$ and a fixed weight function $w: \N_0 \to (0,\infty)$, we let $\big((\mathfrak{B}^*_n, \mathfrak{R}^*_n)\big)_{n\geq 0}$ be the generalized P\'olya urn process corresponding to the sequences $(b^*(i))_{i\geq 0}$ and $(r^*(i))_{i\geq 0}$ defined by
\begin{equation}\label{rb}
\begin{array}{l}
 b^-(i) = w(2i) \\ 
 r^-(i) = w(2i+1)
\end{array},
\qquad 
\begin{array}{l}
 b^+(i) = w(2i+1) \\ 
 r^+(i) = w(2i)
\end{array},
\quad\text{and}\quad  
\begin{array}{l}
 b^0(i) = w(2i) \\ 
 r^0(i) = w(2i)
\end{array},
\quad \text{for } i\geq 0. 
 \end{equation}
With these choices of parameters, at any site $x$ to the left of the origin, the sequence of left/right steps by the random walk on successive visits to $x$ has the same distribution as the sequence of blue/red draws from the generalized P\'olya urn process $\big((\mathfrak{B}^-_n, \mathfrak{R}^-_n)\big)_{n\geq 0}$. Similarly, the process $\big((\mathfrak{B}^+_n, \mathfrak{R}^+_n)\big)_{n\geq 0}$ corresponds to left/right steps at sites to the right of the origin and $\big((\mathfrak{B}^0_n, \mathfrak{R}^0_n)\big)_{n\geq 0}$ corresponds to left/right steps at the origin.  

When considering any of the urn models described above, we let $\tau_k^\mathfrak{B^*}$ (or $\tau_k^\mathfrak{R^*}$) be the number of trials until a blue (or red) ball is selected for the $k$-th time. Namely, setting $\tau_0^\mathfrak{B^*} = \tau_0^{\mathfrak{R^*}}=0$, we have for $k\geq 1$ that 
\begin{equation}\label{taub}
\tau_k^\mathfrak{B^*} = \inf\{ n> \tau_{k-1}^\mathfrak{B^*}:  \mathfrak{B}^*_n = \mathfrak{B}^*_{n-1} + 1 \} 
\  \text{ and } \ 
\tau_k^\mathfrak{R^*} = \inf\{ n> \tau_{k-1}^{\mathfrak{R}^*}:  \mathfrak{R}^*_n = \mathfrak{R}^*_{n-1} + 1 \}. 
\end{equation}
We will also let $\D^*_n$ be the difference between the number of red balls and the number of blue balls after the first $n$ draws from the urn, that is, 
\begin{equation}\label{discr}
 \D^*_n = \mathfrak{R}^*_n - \mathfrak{B}^*_n. 
\end{equation}
For any $n \geq 1$, we denote $\mathcal{F}_n^{\mathfrak{B},\mathfrak{R}}=\sigma((\mathfrak{B}_i,\mathfrak{R}_i),i \leq n)$. 

\section{Extension of results about ${\D}^*$ to the general weight function for PSR walks}\label{sec:ext}

In this section we extend the results proved in \cite[Lemma 4.1 and Section 4.2]{KMP23} for
$w(n)=(n+1)^{-\alpha}$ to the more general form in \eqref{generalw} and without the assumption that $w$ is monotone.

The results below hold for all three of the discrepancy urn processes
$\mathfrak{D}^*$ with $* \in \{+,-,0\}$, and, with the exception of
Proposition \ref{proplimit}, the statements are the same and the
proofs are essentially unchanged in all cases. Thus, for convenience
of notation, for all results below other than Proposition
\ref{proplimit} we will drop the superscript indicating which urn
process is being considered and will give the proofs for the urn
process $\mathfrak{D}^+$ only.

\begin{lemma}[\cite{KMP23}, Lemma 4.1, Remark 4.2] \label{Lemma1}
  There exist constants $C_1,c_1>0$ such that for any integers
  $n,m \geq 1$, 
    \begin{equation}\label{DiscrepTail}
     \P(|\D_{\tau_{n}^{\B}}| \ge m )\le C_1 e^{\frac{-c_1 m^2}{m\vee n}} \quad\text{and} \quad  \P(|\D_{n}| \ge m ) \leq C_1 e^{\frac{-c_1 m^2}{m\vee n}}.
    \end{equation}
\end{lemma}

\begin{proof}
  Since it was noted in \cite[Remark 4.2]{KMP23} that
  $|\D_{2n}| \leq 2 |\D_{\tau_n^{\B}}|$, the second statement in
  \eqref{DiscrepTail} follows immediately from the first (after
  adjusting the constants $C_1$ and $c_1$), and thus we will only prove the first
  statement in \eqref{DiscrepTail}.  We give only the proof for
  $\P(\mathfrak{D}_{\tau_n^{\mathfrak{B}}} \geq m)$ since the
  proof for $\P(\mathfrak{D}_{\tau_n^{\mathfrak{B}}} \leq -m)$
  is similar. Our proof is based on the proof of Lemma 4.1 in
  \cite{KMP23}, but we have to make changes to accommodate the greater
  generality of $w$, including the fact that $w$ is no longer
  assumed to be non-increasing. Recalling Rubin's construction, let
  $(R_j')_{j \geq 1}$ be a sequence of independent exponential random
  variables, independent of $(B_i)_{i \geq 1}$, so that if
  $j \leq n$, $R_j'$ has parameter $r(j-1)$ and if $j > n$, $R_j'$ has
  parameter $\max_{n+1 \leq \ell \leq n+m} r(\ell-1)$.  Then
\begin{equation}\label{eq_lem_4.1}
 \P(\mathfrak{D}_{\tau_n^{\mathfrak{B}}} \geq m) = \P\left(\sum_{i=1}^n B_i \geq \sum_{i=1}^{n+m} R_i\right) \leq \P\left(\sum_{i=1}^n B_i-\sum_{i=1}^{n+m} R_i' \geq 0 \right). 
\end{equation}
We want to study 
\[
 \mathbb{E}\left(\sum_{i=1}^n B_i-\sum_{i=1}^{n+m} R_i'\right)=\sum_{i=0}^{n-1} \left(\frac{1}{w(2i+1)}-\frac{1}{w(2i)}\right)-m \min_{n\leq j \leq n+m-1}\frac{1}{w(2j)}. 
\]
When $n$ is large enough, for all $i \geq n$ we have $\frac{1}{w(i)} \geq \frac{3}{4}i^\alpha$, and 
\begin{equation}\label{Esumdiff}
 \mathbb{E}\left(\sum_{i=1}^n B_i-\sum_{i=1}^{n+m} R_i'\right) \leq \sum_{i=0}^{n-1} \left(\frac{1}{w(2i+1)}-\frac{1}{w(2i)}\right)-\frac{3}{4}m(2n)^\alpha. 
\end{equation}
To evaluate the asymptotics of the sum on the right, note that, for $i\geq 1$,
\[
  \frac{1}{w(2i+1)}-\frac{1}{w(2i)} = \left( (2i+1)^\alpha -
    (2i)^\alpha \right) + 2B \left( (2i+1)^{\alpha-1} -
    (2i)^{\alpha-1} \right) + O\left( i^{\alpha-2} \right).
\]

Moreover, for any $\beta>0$ we have
\begin{align*}
  \sum_{i=0}^{n-1}((2i+1)^\beta-(2i)^\beta)
  &= 1 +  \sum_{i=1}^{n-1}(2i)^\beta\left(\left(1+\frac{1}{2i}\right)^\beta-1\right) = 1 +  \sum_{i=1}^{n-1}(2i)^\beta\left(\frac{\beta}{2i}+O\left(i^{-2}\right)\right) \\
  &=1 + 2^{\beta-1} \sum_{i=1}^{n-1}\left(\beta i^{\beta-1}+O\left(i^{\beta-2}\right)\right) \sim 2^{\beta-1} n^{\beta}, 
\end{align*}
and for $\beta \leq 0$, the same sum starting from $i=1$ is uniformly bounded in absolute value. Therefore, it follows that 
\begin{equation}\label{eq_seriesDL}
 \sum_{i=0}^{n-1} \left( \frac{1}{w(2i+1)} - \frac{1}{w(2i)}  \right) \sim 2^{\alpha-1} n^{\alpha},
\end{equation}
and, thus, for $n$ sufficiently large and all $m\geq 1$
\[
 \mathbb{E}\left(\sum_{i=1}^n B_i-\sum_{i=1}^{n+m} R_i'\right) \leq \left(\frac54 -\frac32 m\right)2^{\alpha-1} n^{\alpha}\le -2^{\alpha-3} m n^\alpha. 
\]

The last inequality and \eqref{eq_lem_4.1} yield 
\[
 \P(\mathfrak{D}_{\tau_n^{\mathfrak{B}}} \geq m) \leq \P\left(\sum_{i=1}^n \frac{B_i-\mathbb{E}(B_i)}{n^\alpha}-\sum_{i=1}^{n+m} \frac{R_i'-\mathbb{E}(R_i')}{n^\alpha} \geq 2^{\alpha-3}m \right).
\]
Note that if $X \sim \mathrm{Exp}(\lambda)$, $\lambda>0$, $t\in\mathbb{R}$, then
$\mathbb{E}(e^{t(X-\mathbb{E}(X))})\leq e^{t^2/\lambda^2}$ for
  all $|t/\lambda|=|t|\mathbb{E}(X)$ small enough. The expectations of
  $(B_i)_{1 \leq i \leq n}$, $(R_i')_{1 \leq i \leq n+m}$ are bounded
  above by
  $\max_{0 \leq \ell \leq 2n } \frac{1}{w(\ell)}\le \bar C_1n^\alpha$ for
  some constant $\bar C_1$. Hence, there exist $t_0>0$ and $g>0$ such that for
  all $|t| \leq t_0$,
\[
 \max_{1 \leq i \leq n}\mathbb{E}\left(e^{t\frac{B_i-\mathbb{E}(B_i)}{n^\alpha}}\right), \max_{1 \leq i \leq n+m}\mathbb{E}\left(e^{t\frac{\mathbb{E}(R_i')-R_i'}{n^\alpha}}\right) \leq e^{\frac{1}{2}gt^2}.
\]
This allows us to use \cite[Theorem III.15]{pSOIRV} and conclude
that there exist a constant $\bar c_1>0$ and $n_0\in\N$ such that
$\P(\mathfrak{D}_{\tau_n^{\mathfrak{B}}} \geq m) \leq e^{-\bar
  c_1 \frac{m^2}{m \vee n}}$ for all $n> n_0$ and all $m \geq 1$.

Finally, let $n\le n_0$. In the spirit of \eqref{eq_lem_4.1}, we 
define $(R_i')_{i \geq 1}$ to be independent exponential random
variables, independent of $(B_i)_{i \geq 1}$, and such that  $R_j'\sim \mathrm{Exp}(r(j-1))$ for
$j \leq n$ and $R_j'\sim \mathrm{Exp}(\lambda)$ for $j>n$ where $\lambda=\max_{\ell\in\mathbb{N}} w(\ell)$. Then for
$t>0$ small,
\begin{align*}
 \P(\mathfrak{D}_{\tau_n^{\mathfrak{B}}} \geq m) 
 &\leq \P\left(\sum_{i=1}^n B_i-\sum_{i=1}^{n+m} R_i' \geq 0 \right)
 =\P\left(-\sum_{i=n+1}^{n+m}R_i' \geq \sum_{i=1}^n (R_i'-B_i)\right)\\
 &\leq \mathbb{E}\left(e^{-t\sum_{i=1}^n (R_i'-B_i)}\mathbb{E}\left(e^{-t\sum_{i=n+1}^{n+m}R_i'}\right)\right) \leq \max_{1 \leq k \leq n_0}\mathbb{E}\left(e^{-t\sum_{i=1}^k (R_i'-B_i)}\right)\left(\frac{\lambda}{\lambda+t}\right)^m. 
\end{align*}
Therefore, there exist constants $\hat C_1,\hat c_1\in(0,\infty)$ such that $\P(\mathfrak{D}_{\tau_n^{\mathfrak{B}}} \geq m) \leq \hat C_1 e^{-\hat c_1 m} \leq \hat C_1 e^{-\hat c_1 \frac{m^2}{m \vee n}}$. This completes the proof. 
\end{proof}

Obtaining estimates on the
variance of $\D_m-\D_n$ directly is not optimal because the increments
$\D_{i+1}-\D_i$ are too strongly correlated. On the other hand, for
$\delta>0$ fixed and $n$ large, the process
$\left(\frac{\D_{\fl{tn}}}{\sqrt{n}} t^\alpha \right)_{t\geq \delta}$
is approximately a martingale, and this observation underlies the
following lemma.

 \begin{lemma}[\cite{KMP23}, Lemma 4.3] \label{lemkey-var} For $\delta\in (0,2)$ fixed, there
   exists a constant $C_2$ such that for $n$ sufficiently large
   and $\delta n \leq k \leq m \leq 2n$ we have
\[
\left| \Var\left( \D_m \left( \frac{m}{n} \right)^\alpha - \D_k \left( \frac{k}{n}\right)^\alpha \right) - n \int_{k/n}^{m/n} u^{2\alpha} \, du \right| 
\leq C_2 \sqrt{n}\log^2 n. 
\]
\end{lemma}

\begin{proof}
We recall that $\mathcal{F}_i^{\mathfrak{B},\mathfrak{R}} := \sigma\left( (\mathfrak{B}_j, \mathfrak{R}_j), \, j\leq i \right)$. The key estimate that needs to be proved is 
 \begin{equation}\label{eq43}
  \mathbb{E}(\mathfrak{D}_{i+1}-\mathfrak{D}_i\mid \mathcal{F}_i^{\mathfrak{B},\mathfrak{R}})=-\alpha\frac{\mathfrak{D}_i}{i}+\varepsilon_i,\text{ with }|\varepsilon_i| \leq \bar C_2 \frac{\mathfrak{D}_i^2+i}{i^2} 
 \end{equation}
 for some constant $\bar C_2$, which was equation (43) in \cite{KMP23}. From \eqref{eq43}, Lemma \ref{Lemma1}, and the inequality $|\D_{2n}| \leq 2|\D_{\tau_n^{\B}}|$, the proof of Lemma \ref{lemkey-var} follows exactly as in \cite{KMP23}, and so we will only prove \eqref{eq43} in our more general setting. In particular, the decomposition \cite[(44)]{KMP23} and the bounds \cite[(45)--(47)]{KMP23} from the proof of \cite[Lemma 4.3]{KMP23} remain valid in our setting for $n>1$; we use them below. 
 
First of all, note that since $|\mathfrak{D}_{i+1} - \mathfrak{D}_i| = 1$ and $|\mathfrak{D}_i| \leq i$, we always have
\[
\left| \mathbb{E}(\mathfrak{D}_{i+1}-\mathfrak{D}_i|\mathcal{F}_i^{\mathfrak{B},\mathfrak{R}}) + \alpha \frac{\mathfrak{D}_i}{i} \right| \leq 1 + \alpha.
\]
Moreover, since $\frac{\mathfrak{D}_i^2 + i}{i^2} \geq \frac{1}{4}$ when $|\mathfrak{D}_i| \geq i/2$, \eqref{eq43} holds with $\bar C_2=4(1+\alpha)$ when $|\mathfrak{D}_i| \geq \frac{i}{2}$. 
 
To prove \eqref{eq43} for  $|\mathfrak{D}_i| \leq \frac{i}{2}$, first note that 
\begin{align}
 \mathbb{E}(\mathfrak{D}_{i+1}-\mathfrak{D}_i\mid \mathcal{F}_i^{\mathfrak{B},\mathfrak{R}}) 
 &= \frac{w(2\mathfrak{R}_i)-w(2\mathfrak{B}_i+1)}{w(2\mathfrak{R}_i)+w(2\mathfrak{B}_i+1)} 
= \frac{w(i+\mathfrak{D}_i)-w(i-\mathfrak{D}_i+1)}{w(i+\mathfrak{D}_i)+w(i-\mathfrak{D}_i+1)}  \nonumber 
  \\
 &=\frac{(i-\mathfrak{D}_i+1)^\alpha(1+O(\frac{1}{i-\mathfrak{D}_i+1}))-(i+\mathfrak{D}_i)^\alpha(1+O(\frac{1}{i+\mathfrak{D}_i}))}{(i-\mathfrak{D}_i+1)^\alpha(1+O(\frac{1}{i-\mathfrak{D}_i+1}))+(i+\mathfrak{D}_i)^\alpha(1+O(\frac{1}{i+\mathfrak{D}_i}))} \nonumber \\
  &=\frac{(1-\frac{\mathfrak{D}_i-1}{i})^\alpha(1+O(\frac{1}{i}))-(1+\frac{\mathfrak{D}_i}{i} )^\alpha(1+O(\frac{1}{i}))}{(1-\frac{\mathfrak{D}_i-1}{i})^\alpha(1+O(\frac{1}{i}))+(1+\frac{\mathfrak{D}_i}{i})^\alpha(1+O(\frac{1}{i}))},  \label{EDidiff}
\end{align}
where in the last equality we used that $O\left( \frac{1}{i-\mathfrak{D}_i+1} \right) = O\left( \frac{1}{i+\mathfrak{D}_i} \right) = O\left(\frac{1}{i} \right)$ when $|\mathfrak{D}_i| \leq \frac{i}{2}$. 
From Taylor expansions
\begin{align*}
 \left( 1- \frac{\mathfrak{D}_i-1}{i} \right)^\alpha
 &= 1-\alpha \frac{\mathfrak{D}_i-1}{i} + O\left( \left( \frac{\mathfrak{D}_i-1}{i} \right)^2 \right) 
 = 1 - \alpha \frac{\mathfrak{D}_i}{i} + O\left( \frac{\mathfrak{D}_i^2 + i}{i^2} \right),\\
\left( 1 + \frac{\mathfrak{D}_i}{i} \right)^\alpha
 &= 1+\alpha \frac{\mathfrak{D}_i}{i} + O\left( \frac{\mathfrak{D}_i^2}{i^2} \right),
\end{align*}
and \eqref{EDidiff} we get that for $|\mathfrak{D}_i| \leq \frac{i}{2}$
\[
 \mathbb{E}(\mathfrak{D}_{i+1}-\mathfrak{D}_i|\mathcal{F}_i^{\mathfrak{B},\mathfrak{R}}) = \frac{-2\alpha\frac{\mathfrak{D}_i}{i}+O(\frac{\mathfrak{D}_i^2+i}{i^2})}{2+O(\frac{\mathfrak{D}_i^2+i}{i^2})} = -\alpha\frac{\mathfrak{D}_i}{i}+O\left(\frac{\mathfrak{D}_i^2+i}{i^2}\right).
\]
This completes the proof of \eqref{eq43}. 
\end{proof}

The following two corollaries follow from Lemmas \ref{Lemma1} and \ref{lemkey-var}, together with (for Corollary \ref{cor-VDn}) the decomposition \cite[(44)]{KMP23} and the bounds \cite[(46)--(47)]{KMP23}, exactly as in \cite{KMP23}.

\begin{corollary}[\cite{KMP23}, Corollary 4.4]\label{cor-VDn1}
 $\lim_{n\to\infty} n^{-1} \Var(\mathfrak{D}_n) = (2\alpha+1)^{-1}$. 
\end{corollary}

\begin{corollary}[\cite{KMP23}, Corollary 4.5]\label{cor-VDn}
There exists a constant $C_3>0$ such that for all $n$ sufficiently large and $m \in [n,2n]$, 
\begin{equation}\label{VDmDn}
\left| \Var(\D_m - \D_n) - (m-n) \right| \leq   C_3\sqrt{n}\log^2 n  +  C_3 (m-n)^2/n.
\end{equation}
\end{corollary}

The proof of the next lemma needs a more significant adjustment, since the argument in \cite{KMP23} used the assumption that $w$ was non-increasing.

 \begin{lemma}[\cite{KMP23}, Lemma 4.6] \label{elenatype} There exist
   $M_0, c_4\in (0, \infty) $ such that for any integer $M \ge M_0 $,
 $n\in\N$, and $y\in[0,\sqrt{n}]$
\[
  \P\bigg(\sup_{\tau^{\B}_{Mn} \le i \leq \tau^{\B}_{(M+1)n}} |\D_i -
    \D_{\tau^{\B}_{Mn}}|\ge y\sqrt{n}\bigg) \le \frac{1}{c_4}\,e^{-c_4
    y^2}.
\]
\end{lemma}

\begin{proof}
Let $\mathcal{G}_{y,M,n}=\{|\mathfrak{D}_{\tau_{Mn}^\mathfrak{B}}| \leq y
\sqrt{Mn}\}$.  Lemma \ref{Lemma1} yields
$\P\left(\mathcal{G}_{y,M,n}^c\right) \leq C_1 e^{-c_1y^2}$; hence it
  suffices to show that
\[
 \P\left( \sup_{\tau_{Mn}^\mathfrak{B} \leq i \leq \tau_{(M+1)n}^\mathfrak{B}}|\mathfrak{D}_i-\mathfrak{D}_{\tau_{Mn}^\mathfrak{B}}| \geq y \sqrt{n} \, \Big| \,  \mathcal{G}_{y,M,n} \right) \leq \frac{1}{c_4} e^{-c_4 y^2}, \quad \text{for } M\geq M_0, \, y \in [0,\sqrt{n}]. 
\]
If there exists an $i$ such that
$\tau_{Mn}^\mathfrak{B} \leq i \leq \tau_{(M+1)n}^\mathfrak{B}$ and
$|\mathfrak{D}_i-\mathfrak{D}_{\tau_{Mn}^\mathfrak{B}}| \geq y
\sqrt{n}$, then let $i_0$ be the smallest such $i$. Since
  $|\mathfrak{D}_{i+1}-\mathfrak{D}_i| = 1$, it must be that
  $|\mathfrak{D}_{i_0}-\mathfrak{D}_{\tau_{Mn}^\mathfrak{B}}| \le
  \lceil y \sqrt{n} \rceil$. Noting that there are at most $n$
  ``down-steps'' of the process $\mathfrak{D}$ between
  $\tau_{Mn}^\mathfrak{B}$ and $i_0$, we conclude that the number of
  ``up-steps'' in the same time interval is at most
  $n+\lceil y \sqrt{n} \rceil$. Therefore,
$i_0 \leq \tau_{Mn}^\mathfrak{B} + 2n+\lceil y \sqrt{n} \rceil$
and we only have to give an appropriate bound
on
\[\P\left(\sup_{\tau_{Mn}^\mathfrak{B} \leq i \leq
      \tau_{Mn}^\mathfrak{B} + 2n+\lceil y \sqrt{n}
      \rceil}|\mathfrak{D}_i-\mathfrak{D}_{\tau_{Mn}^\mathfrak{B}}|
    \geq y \sqrt{n}\, \Big| \ \mathcal{G}_{y,M,n}\right).\] If we let
\[T^+_{y,M,n} =\inf\{i \ge 0:
\mathfrak{D}_{\tau_{Mn}^\mathfrak{B}+i} \geq
\mathfrak{D}_{\tau_{Mn}^\mathfrak{B}}+y\sqrt{n}\}\ \text{ and }\ 
T^-_{y,M,n} =\inf\{i \ge 0:
\mathfrak{D}_{\tau_{Mn}^\mathfrak{B}+i} \leq
\mathfrak{D}_{\tau_{Mn}^\mathfrak{B}}-y\sqrt{n}\},\] then it is enough
to show that
\begin{equation}\label{PTpm}
 \P\left( T^+_{y,M,n} \wedge T^-_{y,M,n} \leq 2n+\lceil y \sqrt{n} \rceil \mid\mathcal{G}_{y,M,n} \right) \leq \frac{1}{c_4} e^{-c_4 y^2}, \quad \text{for } M\geq M_0,\ y \in [0,\sqrt{n}]. 
\end{equation}
In order to do that, we
study the probability for $\left(\mathfrak{D}_{\tau_{Mn}^\mathfrak{B}+i}\right)_{i\geq 0}$ to go up or down, and will prove that for  $i \leq T^+_{y,M,n} \wedge T^-_{y,M,n}$ it is not too far from $1/2$. It is enough to consider $y \in (1,\sqrt{n}]$ since for $y\in[0,1]$ the inequality \eqref{PTpm} holds trivially for all $c_4$ small enough. A computation similar to \eqref{EDidiff} gives that for $i\ge 1$
\begin{align}
 \P(\mathfrak{D}_{i+1}=\mathfrak{D}_i+1\mid \mathcal{F}_i^{\mathfrak{B},\mathfrak{R}})
 &= \frac{ w(2\mathfrak{R}_i)}{ w(2\mathfrak{R}_i) +  w(2\mathfrak{B}_i+1)}
 = \frac{ w(i+\mathfrak{D}_i)}{ w(i+\mathfrak{D}_i) +  w(i-\mathfrak{D}_i+1)} \nonumber \\
 &= \frac{(1-\frac{\mathfrak{D}_i-1}{i})^\alpha(1+O(\frac{1}{i-\mathfrak{D}_i+1}))}{(1-\frac{\mathfrak{D}_i-1}{i})^\alpha(1+O(\frac{1}{i-\mathfrak{D}_i+1}))+(1+\frac{\mathfrak{D}_i}{i})^\alpha(1+O(\frac{1}{i+\mathfrak{D}_i}))}. \label{PDidiff}
\end{align}
If $\tau_{Mn}^\mathfrak{B} \leq i \leq \tau_{Mn}^\mathfrak{B} + T^+_{y,M,n} \wedge T^-_{y,M,n}$ and the event $\mathcal{G}_{y,M,n}$ occurs, then  $|\mathfrak{D}_i| \leq |\mathfrak{D}_i-\mathfrak{D}_{\tau_{Mn}^\mathfrak{B}}|+|\mathfrak{D}_{\tau_{Mn}^\mathfrak{B}}| \leq y \sqrt{n}+1+y\sqrt{Mn} \leq 2y\sqrt{Mn}-1$ if $M$ is large.
Moreover, $i \geq \tau_{Mn}^{\mathfrak{B}} \geq Mn$,
and thus  
\begin{align*}
 \left| \frac{\mathfrak{D}_i-1}{i} \right|, 
 \left| \frac{\mathfrak{D}_i}{i} \right| \leq \frac{2y}{\sqrt{Mn}}
 \quad \text{and}\quad 
 \frac{1}{i-\mathfrak{D}_i+1}, \frac{1}{i+\mathfrak{D}_i} \leq \frac{1}{Mn - 2y\sqrt{Mn}} \leq \frac{2}{Mn}, 
\end{align*}
where the last inequality holds for $y\leq\sqrt{n}$ if $M$ is sufficiently large. 
Using these bounds in \eqref{PDidiff}, a tedious but straightforward calculation yields that there are $k_0,M_1<\infty$ such that for all $1\leq y\leq \sqrt{n}$ and $M\geq M_1$, on the event $\mathcal{G}_{y,M,n}$ we have 
\[
 \left| \P\left(\mathfrak{D}_{i+1}=\mathfrak{D}_i+1 \mid \mathcal{F}_i^{\mathfrak{B},\mathfrak{R}} \right)
 - \frac{1}{2} \right| \leq \frac{k_0 y}{\sqrt{Mn}}, \quad 
 \text{for } \tau_{Mn}^\mathfrak{B} \leq i \leq \tau_{Mn}^\mathfrak{B} + T^+_{y,M,n} \wedge T^-_{y,M,n}.
\]
We can then couple $\big(\mathfrak{D}_{\tau_{Mn}^\mathfrak{B}+i}\big)_{i\geq 0}$ with biased random walks $S^{\pm}$ independent of $\mathcal{G}_{y,M,n}$ with $S^{\pm}_0=0$, $\P(S^{\pm}_{i+1}=S^{\pm}_i+1)=1-\P(S^{\pm}_{i+1}=S^{\pm}_i-1)=\frac{1}{2} \pm \frac{k_0y}{\sqrt{Mn}}$ so that if $\mathcal{G}_{y,M,n}$ occurs, we have
\[
 S^-_i \leq \mathfrak{D}_{\tau_{Mn}^\mathfrak{B}+i}-\mathfrak{D}_{\tau_{Mn}^\mathfrak{B}} \leq S^+_i, \quad \text{for } 
 i\leq T^+_{y,M,n} \wedge T^-_{y,M,n}.
\]
Therefore, 
\begin{align*}
& \P\left( T^+_{y,M,n} \wedge T^-_{y,M,n} \leq 2n+\lceil y \sqrt{n} \rceil \,|\,\mathcal{G}_{y,M,n} \right) \\
 &\quad \leq \P\left(\max_{0 \leq i \leq 2n+\lceil y \sqrt{n} \rceil}S^+_i \geq y \sqrt{n}\right) 
 + \P\left(\min_{0 \leq i \leq 2n+\lceil y \sqrt{n} \rceil}S^-_i \leq -y \sqrt{n}\right) 
 \\
 &\quad \leq 2\P\left(S^+_{2n+\lceil y \sqrt{n} \rceil} \geq y \sqrt{n}\right) 
 + 2\P\left(S^-_{2n+\lceil y \sqrt{n} \rceil} \leq -y \sqrt{n}\right) 
 = 4\P\left(S^+_{2n+\lceil y \sqrt{n} \rceil} \geq y \sqrt{n}\right), 
\end{align*}
where the second inequality follows from the reflection principle, as $\frac{1}{2} - \frac{k_0y}{\sqrt{Mn}}<\frac12<\frac{1}{2}+\frac{k_0y}{\sqrt{Mn}}$, and the last equality follows by symmetry. 
If we choose $M  \geq (16k_0)^2$, we have $\mathbb{E}(S^+_{2n+\lceil y \sqrt{n} \rceil}) = (2n+\lceil y \sqrt{n} \rceil) \frac{2k_0y}{\sqrt{Mn}} \leq \frac{y\sqrt{n}}{2}$, hence by Hoeffding's inequality 
\[
\P\left( T^+_{y,M,n} \wedge T^-_{y,M,n} \leq 2n+\lceil y \sqrt{n} \rceil \,|\,\mathcal{G}_{y,M,n} \right) 
 \leq 4 \P\left(S^+_{2n+\lceil y \sqrt{n} \rceil}-\mathbb{E}(S^+_{2n+\lceil y \sqrt{n} \rceil}) \geq \frac{y\sqrt{n}}{2}\right) \leq 4 e^{-\frac{y^2}{32}}.
\]
This shows \eqref{PTpm} and completes the proof.
\end{proof}

The next two results  follow from Lemmas \ref{Lemma1} and \ref{elenatype} and Corollaries \ref{cor-VDn1} and \ref{cor-VDn} exactly as in \cite{KMP23}. We shall not repeat the proofs.

\begin{lemma}[\cite{KMP23}, Lemma 4.7]\label{DtnD2n-diff}
 There exists a constant $C_5>0$ such that 
 \[
\mathbb{E}\left[ \left( \D_{\tau_n^\B} - \D_{2n} \right)^2 \right] \leq C_5 \sqrt{n}. 
 \]
\end{lemma}

\begin{proposition}[\cite{KMP23}, Proposition 4.8]\label{VarDtmDtn}
  $\displaystyle\lim\limits_{n\to\infty}
    (2n)^{-1}\Var(\D_{\tau_n^{\B}}) = (2\alpha+1)^{-1}$. Moreover,
    there is a constant $C_6>0$ such that for $n$ sufficiently large and
    $m \in [n,2n]$
\[
\left|\Var(\D_{\tau_m^{\B}} - \D_{\tau_n^{\B}}) - 2(m-n) \right|
\leq C_6 n^{3/4} + C_6\frac{(m-n)^2}{n}.
\]
\end{proposition}

Finally, we turn to the asymptotics of the mean of
$\mathfrak{D}_{\tau_n^\mathfrak{B}}$.  To this end, first note that
the next lemma follows from Lemmas \ref{Lemma1} and
\ref{DtnD2n-diff}, the decomposition \cite[(44)]{KMP23}, and the second moment bounds used in the proofs of \cite[(46)--(47)]{KMP23}, exactly as in \cite{KMP23}.
\begin{lemma}[\cite{KMP23}, Lemma 4.9]\label{meanweak}
 $\lim_{n\to\infty} n^{-1/2} \mathbb{E}[ \mathfrak{D}_{\tau_n^\mathfrak{B}} ] = 0.$ 
\end{lemma}

We are now ready to give the limit of
$\mathbb{E}[ \mathfrak{D}^*_{\tau_n^{\mathfrak{B}^*}}]$, and this argument
again needs to be modified  to account for our
more general weight function.  Note that, unlike all of the above results, the limit differs among the three urn processes
$\mathfrak{D}^*$ with $* \in \{-,0,+\}$. The result was stated in
\cite{KMP23} only for $\mathfrak{D}^+$, but here we state and prove the
limit for all three cases.

\begin{proposition}[\cite{KMP23}, Proposition 4.10] \label{proplimit}
\begin{equation}\label{EDtn-lim}
\lim_{n\to\infty} \mathbb{E}[\D^+_{\tau^{\B^+}_n}]
 = \frac{1}{2(2\alpha+1)}, 
 \quad 
 \lim_{n\to\infty} \mathbb{E}[\D^-_{\tau^{\B^-}_n}]
 = \frac{-(4\alpha+1)}{2(2\alpha+1)}, 
 \quad\text{and}\quad 
 \lim_{n\to\infty} \mathbb{E}[\D^0_{\tau^{\B^0}_n}]
 = \frac{-\alpha}{2\alpha+1}.
\end{equation}
\end{proposition}

\begin{proof}
  The proof is similar to that of Proposition
  4.10 of \cite{KMP23}, but, once again, the greater generality of $w$
  forces us to make a few changes. We will use the
  convention that $\sum_{i=n}^{n+m-1}(\cdot)=0$ if $m=0$ and
  $\sum_{i=n}^{n+m-1}(\cdot)=- \sum_{i=n+m}^{n-1}(\cdot)$ if $m<0$. With
    this convention, Lemma 1 of \cite{tGRK} states that
 \[
  \mathbb{E}\left(\sum_{j=0}^{\mathfrak{R}^*_{\tau_n^{\mathfrak{B}^*}}-1}\frac{1}{r^*(j)}\right)=\sum_{j=0}^{n-1}\frac{1}{b^*(j)}. 
 \]
 Since $\mathfrak{R}^*_{\tau_n^{\mathfrak{B}^*}}=\mathfrak{D}^*_{\tau_n^{\mathfrak{B}^*}}+n$, it follows that 
 \begin{equation}\label{EDtn-dec}
  (2n)^\alpha \mathbb{E}(\mathfrak{D}^*_{\tau_n^{\mathfrak{B}^*}})+\mathbb{E}\left(\sum_{j=n}^{\mathfrak{D}^*_{\tau_n^{\mathfrak{B}^*}}+n-1}\left(\frac{1}{r^*(j)}-(2n)^\alpha\right)\right)=\sum_{j=0}^{n-1}\left(\frac{1}{b^*(j)}-\frac{1}{r^*(j)}\right).
 \end{equation}
For the asymptotics of the left side, we first
 notice that, since $(r^*(j))^{-1}=O(j^\alpha)$, there exists a constant $C_7>0$ such that 
 \begin{align*}
  &\left|\mathbb{E}\left(\sum_{j=n}^{\mathfrak{D}^*_{\tau_n^{\mathfrak{B}^*}}+n-1}\left(\frac{1}{r^*(j)}-(2n)^\alpha\right)\mathbbm{1}_{\{|\mathfrak{D}^*_{\tau_n^{\mathfrak{B}^*}}| \geq n^{3/5}\}}\right)\right| \\
  &\qquad\qquad  \leq C_7 \mathbb{E}\left(\left(|\mathfrak{D}^*_{\tau_n^{\mathfrak{B}^*}}|+n\right)^{\alpha+1}\mathbbm{1}_{\{|\mathfrak{D}^*_{\tau_n^{\mathfrak{B}^*}}| \geq n^{3/5}\}}\right) \underset{n \to +\infty}{\longrightarrow} 0,
 \end{align*}
 by Lemma \ref{Lemma1}. To deal with the case $|\mathfrak{D}^*_{\tau_n^{\mathfrak{B}^*}}| < n^{3/5}$, 
 we write 
 \begin{align*}
  \frac{1}{w(2j)} = (2j)^\alpha + O(j^{\alpha-1}) 
  &= (2n)^\alpha\left( 1 + \frac{j-n}{n} \right)^\alpha + O(j^{\alpha-1}) \\
 & = (2n)^\alpha \left\{ 1 + \alpha \frac{j-n}{n} + O\left( \frac{(j-n)^2}{n^2}  \right) \right\}  + O(j^{\alpha-1}),  
 \end{align*}
 and observe that 
 \begin{align*}
  \frac{1}{w(2j)} - \frac{1}{w(2j+1)} & = (2j)^\alpha - (2j+1)^\alpha + O(j^{\alpha-1}) = (2j)^\alpha\left(1-\left(1+\frac{1}{2j}\right)^\alpha\right) + O(j^{\alpha-1}) \\
  &  = (2j)^\alpha\left(1-1-\alpha\frac{1}{2j}+O\left(\frac{1}{(2j)^2}\right)\right) + O(j^{\alpha-1}) = O(j^{\alpha-1}).
 \end{align*}
 Since $r^*(j) = w(2j)$ or $w(2j+1)$, we can conclude that there is a constant $C_7'$ such that
 \[
  \left|\frac{1}{r^*(j)}-(2n)^\alpha-\alpha(2n)^\alpha\,\frac{j-n}{n}\right| 
  \leq C_7'n^{\alpha-\frac{4}{5}}, \quad \text{ for all } |n-j| \leq n^{3/5}.
 \]
It follows that 
\begin{align*}
 &\mathbb{E}\left(\sum_{j=n}^{\mathfrak{D}^*_{\tau_n^{\mathfrak{B}^*}}+n-1}\left(\frac{1}{r^*(j)}-(2n)^\alpha\right)\mathbbm{1}_{\{|\mathfrak{D}^*_{\tau_n^{\mathfrak{B}^*}}| < n^{3/5}\}}\right) \\
 &\qquad = \alpha (2n)^\alpha \mathbb{E}\left( \sum_{j=n}^{\mathfrak{D}^*_{\tau_n^{\mathfrak{B}^*}}+n-1}\left( \frac{j-n}{n} \right) \mathbbm{1}_{\{|\mathfrak{D}^*_{\tau_n^{\mathfrak{B}^*}}| < n^{3/5}\}}\right) + O\left( n^{\alpha-\frac{1}{5}} \right) \\
 &\qquad = \alpha (2n)^\alpha \mathbb{E}\left( \frac{|\mathfrak{D}^*_{\tau_n^{\mathfrak{B}^*}}||\mathfrak{D}^*_{\tau_n^{\mathfrak{B}^*}}-1|}{2n} \mathbbm{1}_{\{|\mathfrak{D}^*_{\tau_n^{\mathfrak{B}^*}}| < n^{3/5}\}}\right) + O\left( n^{\alpha-\frac{1}{5}} \right). 
\end{align*}
Moreover, Lemma \ref{Lemma1}, Proposition \ref{VarDtmDtn}, and Lemma \ref{meanweak} together imply that 
\[
 \mathbb{E}\left( \frac{|\mathfrak{D}^*_{\tau_n^{\mathfrak{B}^*}}||\mathfrak{D}^*_{\tau_n^{\mathfrak{B}^*}}-1|}{2n} \mathbbm{1}_{\{|\mathfrak{D}^*_{\tau_n^{\mathfrak{B}^*}}| < n^{3/5}\}}\right)
 = \frac{\mathbb{E}\left[ \mathfrak{D}^*_{\tau_n^{\mathfrak{B}^*}} \left( \mathfrak{D}^*_{\tau_n^{\mathfrak{B}^*}} - 1 \right) \right]}{2n} + O(n^{-1})
 = \frac{1}{2\alpha+1} + o(1).  
\]
Combining all of the above estimates with \eqref{EDtn-dec}, we have
\begin{equation*}
 (2n)^\alpha \mathbb{E}(\mathfrak{D}^*_{\tau_n^{\mathfrak{B}^*}})
 + \frac{\alpha}{2\alpha+1} (2n)^\alpha + o(n^\alpha) = 
 \sum_{j=0}^{n-1}\left(\frac{1}{b^*(j)}-\frac{1}{r^*(j)}\right).
\end{equation*}
Dividing everything by $(2n)^\alpha$ and taking $n\to \infty$, we get 
\begin{align*}
 \lim_{n\to\infty} \mathbb{E}(\mathfrak{D}^*_{\tau_n^{\mathfrak{B}^*}})
 &= -\frac{\alpha}{2\alpha+1} + \lim_{n\to\infty} \frac{1}{(2n)^\alpha} \sum_{j=0}^{n-1}\left(\frac{1}{b^*(j)}-\frac{1}{r^*(j)}\right) 
 = -\frac{\alpha}{2\alpha+1} 
 +
 \begin{cases}
  \frac{1}{2} & \text{if } *=+ \\
  -\frac{1}{2} & \text{if } *=- \\
  0 & \text{if } *=0, 
 \end{cases}
\end{align*}
where the last equality follows from \eqref{rb} and \eqref{eq_seriesDL}. 
\end{proof}

\bibliographystyle{alpha}

\end{document}